\documentclass[letterpaper, 10 pt, conference]{ieeeconf}  

\IEEEoverridecommandlockouts                              

\usepackage{cite}
\usepackage{amsmath,amssymb,amsfonts}
\usepackage[ruled]{algorithm}
\usepackage{algorithmic}
\usepackage{graphicx}
\usepackage{textcomp}
\usepackage{diagbox}
\usepackage{xcolor}
\usepackage{xspace}

\newcommand{\R}{\mathbb{R}}

\newcommand{\X}{\mathcal{X}}
\newcommand{\T}{\mathbb{T}}

\newcommand{\norm}[1]{\lVert#1\rVert}

\usepackage{scalerel,stackengine}

\newtheorem{lemma}{\textbf{Lemma}}
\newtheorem{theorem}{\textbf{Theorem}}

\newtheorem{remark}{\textbf{Remark}}
\newtheorem{assumption}{\textbf{Assumption}}

\def\BibTeX{{\rm B\kern-.05em{\sc i\kern-.025em b}\kern-.08em T\kern-.1667em\lower.7ex\hbox{E}\kern-.125emX}}

\title{\LARGE \bf A Framework for Time-Varying Optimization\\via Derivative Estimation}

\author{
    Matteo Marchi$^{1}$, Jonathan Bunton$^{1}$, João Pedro Silvestre$^{2}$, and Paulo Tabuada$^{1}$
    \thanks{* This work was funded in part by the Army Research Laboratory under Cooperative Agreement W911NF-17-2-0196.}
    \thanks{$^{1}$Matteo Marchi, Jonathan Bunton, and Paulo Tabuada are with the Electrical and Computer Engineering Department, University of California at Los Angeles, Los Angeles, CA 90095 USA (e-mail: {\tt\small \{matmarchi, j.bunton, tabuada\}@ucla.edu}).}
    \thanks{$^{2}$João Pedro Silvestre is in between IST and UCLA joao.pedro.silvestre@tecnico.ulisboa.pt}%
}

\begin{document}

\maketitle
\thispagestyle{empty}
\pagestyle{empty}

\begin{abstract}
Optimization algorithms have a rich and fundamental relationship with ordinary differential equations given by its continuous-time limit.  When the cost function varies with time -- typically in response to a dynamically changing environment -- online optimization becomes a continuous-time trajectory tracking problem.  To accommodate these time variations, one typically requires some inherent knowledge about their nature such as a time derivative.  

In this paper, we propose a novel construction and analysis of a continuous-time derivative estimation scheme based on ``dirty-derivatives'', and show how it naturally interfaces with continuous-time optimization algorithms using the language of ISS (Input-to-State Stability).  More generally, we show how a simple Lyapunov redesign technique leads to provable suboptimality guarantees when composing this estimator with any well-behaved optimization algorithm for time-varying costs.

\end{abstract}

\section{Introduction}
Optimization problems form the basis of a vast array of engineering problems~\cite{sioshansi2017optimization}. While optimization with respect to a static cost function is a classical field of study, more and more applications are being deployed in dynamic environments, leading to real-time unpredictable changes in the optimization cost function~\cite{dall2020optimization}, and any small improvement in this setting is extremely beneficial to a variety of tasks.


Optimization algorithms are commonly analyzed in the discrete-time setting. However, numerous classical algorithms such as gradient descent, Nesterov acceleration~\cite{muehlebach2019dynamical}, and their variants can be viewed as discretization schemes of continuous-time ODEs (Ordinary Differential Equations) \cite{romero2022ode, cisneros2022contraction}. When modeled as such, properties of these algorithms (like convergence and stability) can be analyzed via the many tools of control and nonlinear systems theory, such as Lyapunov stability~\cite{khalil1992nonlinear}. This view point also impacts the study of time-varying or online optimization. Under some convexity assumptions, time-varying optimization can be analyzed as a \emph{trajectory tracking} problem~\cite{simonetto2020time} by determining how closely the resulting solution $x(t)$ tracks the ideal time-varying minimizer trajectory $x^*(t)=\operatorname{argmin}_{x}f(x,t)$, a widely studied problem in continuous-time control theory.

In the time-varying setting, the literature generally takes one of two approaches. The first is to pick an algorithm developed for the time-invariant case and proving that it is robust to some level of time-variation of the cost function. In \cite{popkov2005gradient} the author derives tracking error bounds for the classical gradient descent method, and \cite{yang2016tracking} gives regret bounds in the Online Convex Optimization (OCO) setting under slow variations. The second approach, instead, tries to redesign or adapt an optimization method by directly incorporating knowledge of the time variation. In continuous-time, multiple works add a time-varying correction term to the continuous-time analogue of Newton's method, as under ideal conditions that fully compensates for the time variation \cite{simonetto2020time, rahili2015distributed, gong2016time}, with some works also analyzing time-varying constraints~\cite{fazlyab2017prediction}.

As one might expect, incorporating knowledge of the time variations in the cost can often bring massive performance/stability benefits. However, this typically requires having direct knowledge of the time-variation in the form of a time derivative, such as $\frac{\partial}{\partial t}\nabla_x f(x,t)$. This may not constitute a problem in some applications, such as a control task where we want to track a known reference trajectory, but many others do not allow for this kind of knowledge ahead of time. In~\cite{simonetto2016class}, this time derivative is approximated via finite differences, but their analysis is tailored to their specific discrete-time algorithm and hard to generalize. Another option, pursued in this work, is to study a general derivative estimation technique, and under what conditions the interconnection of such an estimator with an optimization algorithm results in a well-behaved system.

Derivative estimation has a rich history, from numerous variants of finite difference~\cite{khan1999closed, patidar2005use, hassan2012algorithm} to interpolation-based~\cite{krishnan2012selection} schemes. Notably, many observers and state estimation methods from control theory can be adapted to perform derivative estimation (related to the notion of flatness for dynamical systems~\cite{isidori1986sufficient}), of which high-gain observers~\cite{khalil2014high, astolfi2018low} are a notable example. The so-called dirty-derivative originates from PID control~\cite{johnson2005pid, loria2015observers} as a common approximation to implement a ``derivative'' term in the control signal, but has received little interest from a theoretical point of view. This filter is the basis of the extended derivative estimator we introduce in this work. The ability of this construction to interface with the framework we develop is crucial for our analysis.

Our contribution is summarized as follows: 1) We introduce a framework to adapt continuous-time static optimization algorithms to include time-varying knowledge, and establish when this is possible through the language of ISS (Input-to-State Stability) theory~\cite{Sontag2008}. 2) We provide a novel construction of a general derivative estimator based on dirty-derivatives, and we derive explicit IOS (Input-to-Output Stability)~\cite{sontag1999notions} estimation error bounds. 3) We establish that the dynamical system resulting from the interconnection of such an adapted optimization algorithm and our derivative estimator results, also, is an IOS system. This translates to robust behavior and bounds on the time-varying minimizer tracking error $\norm{x(t) - x^*(t)}$.

\section{Problem Statement}\label{sec:problem}
We consider a convex optimization task where the cost function varies with time.  More explicitly, consider a feasible set $\X\subseteq\R^n$, a set of parameter vectors $\Theta \subseteq \R^{p}$ and a time domain $\T\subseteq\R_{\geq 0}$, alongside a cost function $f: \X \times\Theta \to \R$. Assuming that the parameters change with time according to some unknown function $\theta: \T\to\Theta$, we seek to design an Ordinary Differential Equation (ODE):
\begin{align}\label{eq:optimization_ode}
    \dot{x} &= g(x) + u,
\end{align}
with $x\in\X$, $g:\X\to\R^n$, and $u:\T\to\R^n$ such that its solution, $x(t)$, minimizes the time-varying function $f$ for all $t\in\T$. Because the time-varying law of the parameters is unknown, $g$ and $u$ cannot be designed using knowledge of differential properties of $\theta$, such as its time derivative $\dot\theta(t)$, but we may use its current value $\theta(t)$.


We make the following common assumptions:
\begin{assumption}\label{ass:strong_convexity}
    The function $f$ is smooth and strongly convex for every fixed $\overline\theta\in\Theta$. I.e., there exists $\mu\in\R_{>0}$ such that:
    \begin{align*}
        f(x, \overline\theta) &\geq f(x', \overline\theta) + \nabla_x f(x', \overline\theta)^\top (x-x') + \frac{\mu}{2}\Vert x-x'\Vert_2^2,
    \end{align*}
    for all $x, x'\in\X$, and $\theta\in\Theta$.
\end{assumption}

Under these conditions, there exists a (unique) minimizing trajectory described as $x^*(\theta(t)) = \operatorname{argmin}_{x\in\mathcal X} f(x,\theta(t))$, and we then seek to minimize a \textit{tracking error} with respect to this trajectory. I.e., we design $g$ and $u$ in \eqref{eq:optimization_ode} such that $x-x^*$ converges to a neighborhood of the origin, the size of which may depend on the magnitude of the time-variations in $f$.

\section{Online Optimization}
Classically, continuous-time algorithms in the form \eqref{eq:optimization_ode} arise as the continuous limit of iterative algorithms, with the prime example being the \emph{gradient flow} $\dot x = \nabla_x f(x)$.  Multiple works have shown the effectiveness of gradient flows, even in online optimization tasks, when the time variations are slow \cite{popkov2005gradient}.

A related family of methods directly incorporates knowledge of the time-variations in the cost function.  A typical candidate for this approach is the continuous limit of Newton's method:
\begin{align}\label{eq:newton_ode}
    \dot{x} &= -\left(\nabla_{xx} f(x, \theta)\right)^{-1}\nabla_x f(x, \theta) + u.
\end{align}
Here $\nabla_{xx}f(x, \theta)$ denotes the Hessian matrix associated with a time-varying $f(x, \theta)$. When augmenting \eqref{eq:newton_ode} with time variation knowledge, one often considers the correction:
\begin{align}\label{eq:newton_correction}
    u(t) &= -\left(\nabla_{xx} f(x, \theta)\right)^{-1}\nabla_{x\theta}f(x, \theta)\dot\theta,
\end{align}
where $\nabla_{x\theta}f(x, \theta)$ is the partial derivative of the gradient of $f$ with respect to $\theta$. For a detailed analysis of~\eqref{eq:newton_ode}-\eqref{eq:newton_correction}, see~\cite{fazlyab2017prediction}.


Traditionally, the performance of \eqref{eq:optimization_ode} is analyzed purely through the lens of stability, but here we consider the slightly more general view of \emph{input-to-state stability} (ISS)~\cite{Sontag2008}. For a fixed $\theta$, consider the corresponding minimizer $x^*(\theta)$. Then, we say $\eqref{eq:optimization_ode}$ is ISS with respect to $u$ if:
\begin{align}\label{eq:iss_condition}
    \Vert x(t)-x^*(\theta) \Vert \leq \beta(\Vert  x(t_0)-x^*(\theta) \Vert, t-t_0) + \kappa(\Vert u(t) \Vert_\infty),
\end{align}
for all $t_0\leq t\in\T$, where $\kappa$ is a class $\mathcal{K}$ function, $\beta$ is a class $\mathcal{KL}$ function, $\Vert\cdot\Vert$ denotes the usual vector 2-norm, and $\Vert u(t)\Vert_\infty$ is the essential supremum of $u$ over $[t_0,t]$.\footnote{A class $\mathcal{K}$ function $\kappa:\R_{\ge 0}\to\R_{\ge 0}$ is strictly increasing with $\kappa(0) = 0$.  A class $\mathcal{KL}$ function $\beta:\mathbb{R}\times\mathbb{R}\to\mathbb{R}$ is a continuous function such that $\beta(\cdot,s)\in\mathcal{K}$ and $\beta(r,\cdot)$ is strictly decreasing with $\lim_{s\to\infty}\beta(r,s) = 0$.}

ISS has an equivalent Lyapunov characterization. Equation $\eqref{eq:iss_condition}$ is equivalent to the existence of a radially unbounded function $V:\R^n\times\R^p\to\R_{\geq 0}$ with $V(x^*(\theta),\theta) = 0$ and $\alpha_1,\alpha_2,\alpha_3,\alpha_4\in\mathcal K$ that satisfy:
\begin{equation}\label{eq:iss_lyapunov}
\begin{gathered}
    \alpha_1(\Vert x-x^* \Vert) \leq V(x,\theta) \leq \alpha_2(\Vert x-x^* \Vert), \\
    \dot V(x,\theta) = \nabla_xV(x,\theta)\dot x \leq -\alpha_3(V(x, \theta)) + \alpha_4(\Vert u(t)\Vert).
\end{gathered}
\end{equation}
Note that when $u = 0$, the system is (uniformly) globally asymptotically stable.

We observe that, under some mild conditions, any ODE 
\eqref{eq:optimization_ode} that satisfies this ISS property for the family of \emph{time-invariant} optimization problems, can be adapted into a robust ODE for the \emph{time-varying} optimization problem, and formalize this assertion.
\begin{lemma}\label{lem:lyapunov_redesign}
    Consider system~\eqref{eq:optimization_ode} and suppose $g$ is such that for any fixed value of $\theta\in\Theta$ and its corresponding minimizer $x^*(\theta)$, system~\eqref{eq:optimization_ode} is ISS with respect to $u$. Let $V(x,\theta)$ be a corresponding family of Lyapunov functions differentiable with respect to $\theta$ with bounded $\norm{\nabla_\theta V(x,\theta)}<\infty$. Then, if $\theta(t)$ is a smooth function of time, any continuous $u:\T\to\R^n$ satisfying:
    \begin{align}\label{eq:input_condition}
        \nabla_xV(x,\theta) u + \nabla_\theta V(x,\theta)(\dot\theta + d) \leq 0,
    \end{align}
     renders \eqref{eq:optimization_ode} ISS with respect to the signal $d:\T\to\R^p$ when considering the time-varying minimizer $x^*(\theta(t))$. I.e., the following inequality holds for some $\beta\in\mathcal{KL}$ and $\kappa\in\mathcal K$:
    \begin{equation}\label{eq:lemma_optimization_iss}
        \big\Vert x(t)-x^*(\theta(t))\big\Vert \le \beta\left(\big\Vert x(0)-x^*(\theta(0))\big\Vert, t\right) + \kappa(\norm{d}_\infty).
    \end{equation}
\end{lemma}

In later sections the signal $d$ will be used to model the estimation error on $\dot\theta$.

\begin{proof}
    We follow the same basic steps as in Lyapunov redesign for ISS~\cite{Sontag2008}.
    Taking the time derivative of $V(x,\theta$) we have:
     \begin{align*}
        \dot{V}(x, \theta) &= \nabla_xV(x,\theta)(g(x) + u) + \nabla_\theta V(x,\theta)\dot\theta \\
        &= \nabla_xV(x,\theta)g(x) + \nabla_xV(x,\theta)u + \nabla_\theta V(x,\theta)\dot\theta \\
        &\leq -\alpha_3(V(x,\theta)) + \nabla_xV(x,\theta)u + \nabla_\theta V(x,\theta)\dot\theta \\
        &\leq -\alpha_3(V(x,\theta)) - \nabla_\theta V(x,\theta)d \\
        &\leq -\alpha_3(V(x,\theta)) +  \alpha_d(\Vert d\Vert),
    \end{align*}
    where we have used \eqref{eq:iss_lyapunov} in the first inequality and defined $\alpha_d(\Vert d \Vert) = \sup_{x,\theta} (\Vert \nabla_\theta V(x,\theta)\Vert) \Vert d \Vert$. Then, the existence of $\beta$ and $\kappa$ satisfying \eqref{eq:lemma_optimization_iss} immediately follows by equivalence with the Lyapunov characterization of ISS. 
\end{proof}


This approach applies generically to any continuous-time algorithm that is asymptotically stable for the entire class of \emph{static} problems encountered, such as strongly convex costs.  In this way, Lemma \ref{lem:lyapunov_redesign} provides a framework for adapting a sufficiently well-behaved continuous-time algorithm to the time-varying scenario.
\begin{remark}
    A choice of $u$ satisfying \eqref{eq:input_condition} always exists as long as $\nabla_x V(x,\theta) \ne 0$ whenever $\nabla_\theta V(x,\theta) \ne 0$. As a special case, this is always satisfied for the choice\footnote{A simple computation shows that $\nabla_x V(x,\theta) = \nabla_{xx} f(x,\theta)\nabla_x f(x,\theta)$ and $\nabla_\theta V(x,\theta) = \nabla_{x\theta} f(x,\theta)\nabla_x f(x,\theta)$. Therefore, $\nabla_{\theta} V \ne 0$ implies $\nabla_x f \ne 0$ which in turns implies $\nabla_x V \ne 0$ because of strong convexity of $f$.} $V(x,\theta) = \frac{1}{2}\norm{\nabla_x f(x,\theta)}^2$. Under Assumption~\ref{ass:strong_convexity} on $f$, this is a Lyapunov function for~\eqref{eq:newton_ode} and the correction~\eqref{eq:newton_correction} always satisfies condition \eqref{eq:input_condition} with $d=0$.
\end{remark}
\begin{remark}
    If condition~\eqref{eq:input_condition} is not always satisfied, we can still obtain a weaker version of Lemma~\ref{lem:lyapunov_redesign} as long as there is a finite threshold $\overline V\in\R_{\ge 0}$ such that $V(x,\theta) > \overline V$ implies $\nabla_x V(x,\theta) \ne 0$. In that case the result reduces to ISpS (input-to-state practical stability).
\end{remark}


\section{Derivative Estimation}\label{sec:main_results}
In Lemma \ref{lem:lyapunov_redesign}, we showed how online optimization methods can be designed by exploiting knowledge of the time variations in the function.  Unfortunately, the class of control inputs in \eqref{eq:input_condition} is characterized by explicit knowledge of $\dot\theta$, which is generally unavailable in the online scenario. To remedy this issue, we consider an extension of the classical continuous-time differentiation scheme, sometimes called \textit{dirty-derivative}~\cite{loria2015observers}, and provide a proof of convergence.

The dirty-derivative operator approximates the derivative of a signal $w:\T \to \R$ and is described by the transfer function:
\begin{equation}\label{eq:dd_simple}
    \frac{\sigma s}{s + \sigma},
\end{equation}
where $\sigma\in\R_{>0}$ determines how fast the filter defined by the transfer function tracks the derivative of $w$.

We generalize~\eqref{eq:dd_simple} to estimate derivatives up to some order $k$, additionally resulting in a flexible choice of error bounds. The proposed estimator follows a recursive structure, where $\widehat W_{i}(s)$ is the Laplace transform of the $i$-th estimated derivative for $i\in\{1,2,\dots,k\}$:
\begin{equation}
\begin{aligned}\label{eq:dirty_der_recursive}
    \widehat W_i(s) &=
        \frac{\sigma^is^i}{(s+\sigma)^i}W(s) + \frac{(s+\sigma)^i-\sigma^i}{s(s+\sigma)^i}\widehat W_{i+1}(s)\\
    \widehat W_k(s) &= \frac{\sigma^ks^k}{(s+\sigma)^k}W(s).
\end{aligned}
\end{equation}
A block diagram of~\eqref{eq:dirty_der_recursive} for $k=3$ is shown in Fig.~\ref{fig:dirty_diagram}. Each derivative estimated by~\eqref{eq:dirty_der_recursive} converges exponentially to a ball around the true derivative whose size is a function of the magnitude of the $(k+1)$-th derivative of $w$.

\begin{figure}[h]
    \centering
    \includegraphics[width=0.5\textwidth]{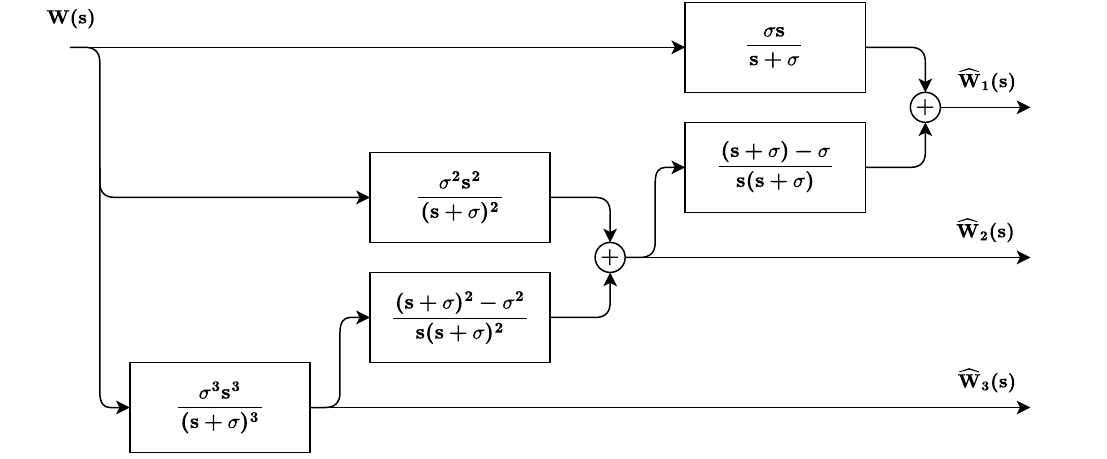}
    \caption{System~\eqref{eq:dirty_der_recursive} estimating the first three derivatives of a signal $U(s)$.}
    \label{fig:dirty_diagram}
    \vspace{-5mm}
\end{figure}

We are now in the position to prove our previous claim for the dirty-derivative-based system~\eqref{eq:dirty_der_recursive}. Note that we use $w^{(i)}$ to denote the $i$-th derivative of $w$.
\begin{theorem}\label{theo:k_dd_diff_operator}
    Let $w:\R_{\ge 0}\to\R$ be a $k+1$ times differentiable input signal and $\widehat w_{i}:\R_{\ge 0}\to\R$ for $i=\{1,2,\dots,k\}$ be the $i$-th estimate produced by~\eqref{eq:dirty_der_recursive} when fed $w$. Then, for all $t \ge 0$:
    \begin{equation} \label{eq:k_dd_guarantee}
        \big\Vert \widehat w_i(t)-w^{(i)}(t)\big\Vert \leq \beta_{i,\sigma}e^{-\gamma_i\sigma t} + \frac{\alpha_i}{\sigma^{k+1-i}}\big\Vert w^{(k+1)}(t)\big\Vert_\infty,
    \end{equation}
    for some $\beta_{i,\sigma}\in\R_{>0}$ dependent on initial conditions, and $\gamma_i, \alpha_i\in\R_{>0}$.
    In particular:
    \begin{equation}
        \lim_{t\to\infty}\norm{\widehat w_i(t) - w^{(i)}(t)} \leq \frac{\alpha_i}{\sigma^{k+1-i}}\norm{w^{(k+1)}(t)}_\infty.
    \end{equation}

    
\end{theorem}
\begin{proof}
    Let us define the estimation error of the $i$-th derivative of $w$ as $e_i = \hat w_i - w^{(i)}$. Then, for $i=k$ the Laplace transform of the error is: 
    \begin{equation}
        \begin{aligned}
            E_k &= \widehat W_k - s^kW = \frac{\sigma^ks^k}{(s+\sigma)^k}W - s^kW\\
            &= \frac{\sigma^k-(s+\sigma)^k}{(s+\sigma^k)}s^kW = -\frac{(s+\sigma)^k-\sigma^k}{s(s+\sigma^k)}s^{k+1}W\\
            &= -F_k(s)s^{k+1}W,
        \end{aligned}
    \end{equation}
    where we define $F_i(s)$ for $i\in\{1,2,\dots,k\}$ as:
    \begin{align}
        F_i(s) &= \frac{(s+\sigma)^i-\sigma^i}{s(s+\sigma)^i}.
    \end{align}
    For $i<k$ the error is:
    \begin{equation}
        \begin{aligned}
            E_i &= \widehat W_i - s^iW\\
            &= \frac{\sigma^is^i}{(s+\sigma)^i}W + \frac{(s+\sigma)^i-\sigma^i}{s(s+\sigma)^i}\widehat W_{i+1} - s^iW\\
            &= \frac{\sigma^i-(s+\sigma)^i}{(s+\sigma^i)}s^iW + \frac{(s+\sigma)^i-\sigma^i}{s(s+\sigma)^i}\widehat  W_{i+1}\\
            &= \frac{(s+\sigma)^i-\sigma^i}{s(s+\sigma)^i}\left(\widehat W_{i+1} - s^{i+1}W\right)\\
            &= \frac{(s+\sigma)^i-\sigma^i}{s(s+\sigma)^i}E_{i+1}\\
            &= F_i(s)E_{i+1}.
        \end{aligned}
    \end{equation}
    Because the transfer function $F_i(s)$ corresponds to the state-space system in Lemma~\ref{lem:single_k_dd_state_space_appendix} (see the appendix\footnote{Inspecting the system defined in Lemma~\ref{lem:single_k_dd_state_space_appendix}, this can be easily verified by computing $C_n(sI-A_n)^{-1}B_n$. Because $A_{n,\sigma}$ is in controllable canonical form, $\det{(sI-A_n)^{-1}} = (s+\sigma)^n$, and due to $C_n$ and $B_n$, we only need to compute the top left cofactor of $sI-A_n$, which equals $s^{-1}\left((s+\sigma)^n-\sigma^n\right)$.}), we can apply its result recursively (denoting the internal state of $F_i$ as $x_{e_i}$):
    \begin{equation*}
        \begin{aligned}
            \norm{e_i(t)} &\leq \norm{x_{e_i}(0)}b_{i,\sigma}e^{-d_i\sigma t} + \frac{a_{i}}{\sigma}\norm{e_{i+1}(t)}\\
            &\leq \norm{x_{e_i}(0)}b_{i,\sigma}e^{-d_i\sigma t} \\&\quad+ \frac{a_{i}}{\sigma}\left(\norm{x_{e_{i+1}}(0)}b_{i+1,\sigma}e^{-d_{i+1}\sigma t} + \frac{a_{i+1}}{\sigma}\norm{e_{i+2}(t)}\right)\\
            &\vdots \\
            &\leq \sum_{j=i}^{k}\left(\left(\prod_{\ell=i}^{j-1}\frac{a_{\ell}}{\sigma}\right)\norm{x_{e_{j}}(0)}b_{j,\sigma}e^{-d_j\sigma t}\right) \\&\quad+ \left(\prod_{j=i}^{k}a_{j}\right)\frac{\norm{w^{(k+1)}(t)}_\infty}{\sigma^{k+1-i}}.
        \end{aligned}
    \end{equation*}
    Because the first term is a sum of exponentially decaying terms with a common $\sigma$ in the exponent, there exist $\beta_{i,\sigma}, \gamma_i, \alpha_i\in\R_{>0}$ such that:
    \begin{equation}
    \begin{gathered}
        \sum_{j=i}^{k}\left(\left(\prod_{\ell=i}^{j-1}\frac{a_{\ell}}{\sigma}\right)\norm{x_{e_{j}}(0)}b_{j,\sigma}e^{-d_j\sigma t}\right) \leq \beta_{i,\sigma}e^{-\gamma_i\sigma t},\\
        \left(\prod_{j=i}^{k}a_{j}\right)\frac{\norm{w^{(k+1)}(t)}_\infty}{\sigma^{k+1-i}} \le \alpha_i\frac{\norm{w^{(k+1)}(t)}_\infty}{\sigma^{k+1-i}}.
    \end{gathered}
    \end{equation}
\end{proof}
\begin{remark}
Although this result is for a scalar signal, it immediately generalizes to estimating $k$-derivatives of an $m$-dimensional signal $w:\T\to\R^m$ by running $m$ copies of the estimator in parallel; one for each scalar component of $w$.
\end{remark}
\begin{remark}
    The asymptotic error bound in Theorem~\ref{theo:k_dd_diff_operator} approaches zero for $\sigma\to\infty$, but with a rate that scales differently with the order of derivative $i\leq k$.  In particular, higher-order derivatives are more difficult to estimate, and the size of the ball to which we are guaranteed to converge scales inversely with powers of the derivative order.
\end{remark}
\begin{remark}
    In the special case where the signal $w:\T\to\R$ is described by a $k$-degree polynomial, the $(k+1)$-th derivative vanishes, and $\big\Vert w^{(k+1)}(t)\big\Vert_\infty = 0$. Theorem \ref{theo:k_dd_diff_operator} then guarantees that the estimates converge exponentially to the exact values of the derivatives of $w$ for any choice of $\sigma$.
\end{remark}

\section{Framework for Time-varying Optimization}
Both the online optimization algorithms designed via Lemma \ref{lem:lyapunov_redesign} and the dirty derivative scheme in Theorem~\ref{theo:k_dd_diff_operator} come with ISS-style convergence guarantees (specifically, Theorem~\ref{theo:k_dd_diff_operator} gives an IOS (input-to-output stability) result~\cite{sontag1999notions}). Because of these results, we can cascade these systems and preserve their convergence properties. In particular, Lemma~\ref{lem:lyapunov_redesign} shows that a well-behaved (in the ISS sense) time-invariant optimization algorithm can be adapted into a time-varying one assuming knowledge of $\dot\theta$. We now present our main contribution, and show that the estimates produced by our dirty-derivative scheme~\eqref{eq:dirty_der_recursive} can be used in place of $\dot\theta$ while still preserving input-to-output stability of the interconnected system. We formalize this notion in the following theorem.


\begin{theorem}\label{cor:final}
    Consider a system~\eqref{eq:optimization_ode} satisfying the assumptions of Lemma \ref{lem:lyapunov_redesign}, and let $\widehat\theta_1:\T\to\Theta$ be the estimate of $\dot\theta$ provided by a dirty-derivative estimator~\eqref{eq:dirty_der_recursive} of order $k$ when fed $\theta$ as input. Then, if $u$ is selected such that:
    \begin{align}\label{eq:input_condition_est}
        \nabla_xV(x,\theta)u + \nabla_\theta V(x,\theta)\widehat\theta_1 \leq 0,
    \end{align}
    the combination of system~\eqref{eq:optimization_ode} and the estimator is IOS (input-to-output stable) with respect to the signal $\sigma^{-k}\theta^{(k+1)}(t)$ when considering the time-varying minimizer $x^*(\theta(t))$. I.e., with outputs $\widetilde x(t) = x(t)-x^*(\theta(t))$ and $e_1 = \widehat\theta_1-\dot\theta$  there exist $\beta'_\sigma\in\mathcal{KL}$ and $\kappa'\in\mathcal K$ such that:
    \begin{equation}\label{eq:main_contribution}
    \begin{aligned}
        &\Big\Vert\Big(\widetilde x(t), e_1(t)\Big)\Big\Vert \le \\&\quad \beta'_\sigma\left(\Big\Vert\Big(\widetilde x(0), x_e(0)\Big)\Big\Vert, t\right) + \kappa'\Big(\sigma^{-k}\big\Vert\theta^{(k+1)}\big\Vert_\infty\Big),
    \end{aligned}
    \end{equation}
    where $x_e$ denotes the full internal state of the dirty-derivative error system.
    
    
\end{theorem}
\begin{proof}
    We first observe that by~\eqref{eq:input_condition_est}:
    \begin{equation}
    \begin{aligned}
        0 &\ge \nabla_xV(x,\theta)u + \nabla_\theta V(x,\theta)\widehat\theta_1\\
        &= \nabla_xV(x,\theta)u + \nabla_\theta V(x,\theta)(\dot\theta + (\widehat\theta_1 - \dot\theta))\\
        &= \nabla_xV(x,\theta)u + \nabla_\theta V(x,\theta)(\dot\theta + e_1).
    \end{aligned} 
    \end{equation}
    Then, by applying Lemma~\ref{lem:lyapunov_redesign} with $d = e_1$ we can write:
    \begin{equation*}
        \big\Vert x(t)-x^*(\theta(t))\big\Vert \le \beta\left(\big\Vert x(0)-x^*(\theta(0))\big\Vert, t\right) + \kappa(\norm{e_1}_\infty),
    \end{equation*}
    and:
    \begin{equation*}
    \begin{aligned}
        \Big\Vert\Big(\widetilde x(t), e_1(t)\Big)\Big\Vert &\le \Vert \widetilde x(t)\Vert + \norm{e_1(t)}\\
        &= \big\Vert x(t)-x^*(\theta(t))\big\Vert + \norm{e_1(t)}\\
        &\le \beta\left(\big\Vert x(0)-x^*(\theta(0))\big\Vert, t\right) + \kappa(\norm{e_1}_\infty)\\&\qquad + \beta_{1,\sigma}e^{-\gamma_1\sigma t} + \frac{\alpha_1}{\sigma^{k}}\big\Vert\theta^{(k+1)}(t)\big\Vert_\infty.
    \end{aligned}
    \end{equation*}
    By the nonlinear superposition principle~\cite[Section 3.1]{Sontag2008}, only the essential supremum norm for $t\to\infty$ matters, so we can substitute the essential supremum norm $\norm{e_1(t)}_\infty$ by a limsup $\overline{\lim}_{t\to\infty} \norm{e_1(t)} \le \frac{\alpha_1}{\sigma^{k}}\big\Vert\theta^{(k+1)}(t)\big\Vert_\infty$:
    \begin{equation*}
    \begin{aligned}
        &\Big\Vert\Big(\widetilde x(t), e_1(t)\Big)\Big\Vert \le \Big[\beta\left(\big\Vert \widetilde x(0)\big\Vert, t\right) + \beta_{1,\sigma}e^{-\gamma_1\sigma t}\Big]\\&\qquad\qquad + \left[\kappa\left(\frac{\alpha_1}{\sigma^{k}}\big\Vert\theta^{(k+1)}(t)\big\Vert_\infty\right) + \frac{\alpha_1}{\sigma^{k}}\big\Vert\theta^{(k+1)}(t)\big\Vert_\infty\right].
    \end{aligned}
    \end{equation*}
    By composition of class $\mathcal K$ functions, and because $\beta_{1,\sigma}e^{-\gamma_1\sigma t}$ is monotonically decaying with time with $\beta_{1,\sigma}$ dependent on the initial conditions of the internal error system's state, there exist $\beta'_{\sigma}\in\mathcal{KL}$ and $\kappa'\in\mathcal K$ such that~\eqref{eq:main_contribution} holds.

\end{proof}

Theorem~\ref{cor:final}, together with Lemma~\ref{lem:lyapunov_redesign} and our dirty-derivative estimator, provide a general framework to adapt a continuous time-invariant optimization algorithm to an online time-varying algorithm. Notably, explicit time derivative knowledge of $\theta$ is not necessary, since we can approximate it while retaining the stability properties of the ideal algorithm, and the performance of the optimization algorithm can be improved arbitrarily by increasing the gain $\sigma$. However, in practice the presence of measurement noise on $\theta$, or a discrete implementation of the continuous time algorithm will limit how high $\sigma$ can be set. The effect of noise on our proposed construction is explored empirically in the following section.


\vspace{-2mm}
\section{Experiments}
In this section we demonstrate in simulation the effectiveness of our proposed derivative estimation scheme~\eqref{eq:dirty_der_recursive}, and its application to a time-varying optimization task where we use it to estimate the time-varying correction in Lemma~\ref{lem:lyapunov_redesign}.

\subsection{Derivative Tracking}
We first show the system~\eqref{eq:dirty_der_recursive} tracking the first derivative of a sinusoidal signal $w(t) = \sin\left(5t-2\right)$ for $\sigma\in\{5, 20\}$ (see Fig.~\ref{fig:exisgnal}).
\begin{figure}[h]
    \centering
    \vspace{-3mm}
    \includegraphics[width=0.5\textwidth]{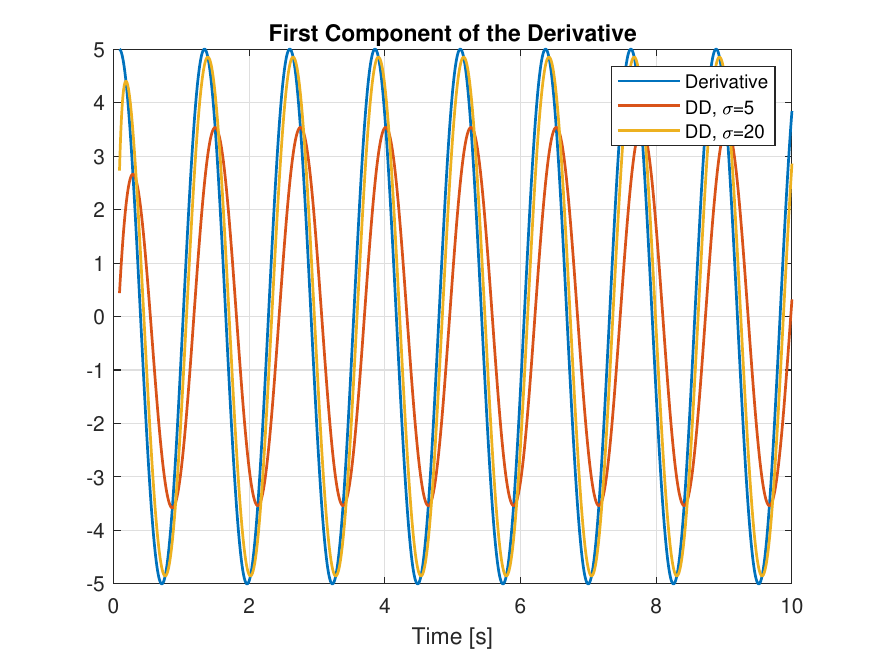}
    \vspace{-5mm}
    \caption{Plot of the first derivative of the signal $w(t)=\sin(5t-2)$ and its estimates for $\sigma\in\{5,20\}$.}
    \label{fig:exisgnal}
\end{figure}
As expected, the tracking is significantly better for the bigger value of $\sigma$. In a practical scenario, we may only have access to a corrupted version of $w$, perturbed by some noise $n(t)$, i.e., $\widetilde{w}(t)=w(t)+n(t)$, and we want our estimation to reject this noise effectively. Noise rejection is practically relevant even when the signal is fully measurable, as a typical time-discretization of~\eqref{eq:dirty_der_recursive} is equivalent to the continuous system working with bounded noise.

In Fig.~\ref{fig:exsignalnoise}, we show first-derivative tracking of the same signal $w(t) = \sin\left(5t-2\right)$ for $\sigma\in\{5, 20\}$ corrupted by zero mean Gaussian noise with $\operatorname{var}\left(n\right)=0.01$.
\begin{figure}[h]
    \centering
    \includegraphics[width=0.5\textwidth]{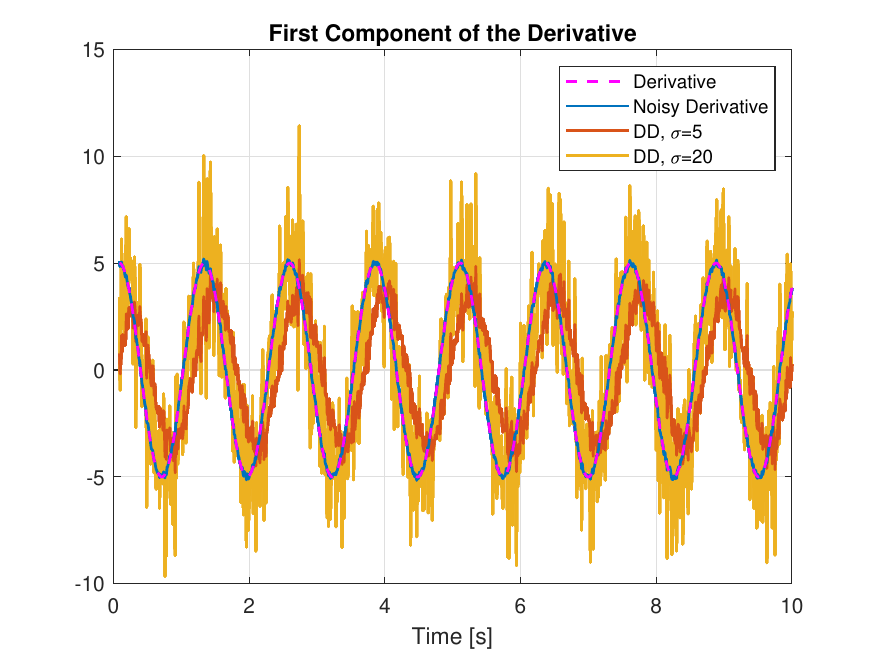}
    \caption{Plot of the first derivative of the signal $w(t)=\sin(5t-2)$ and its estimates for $\sigma\in\{5,20\}$ when $w$ is corrupted by Gaussian noise.}
    \label{fig:exsignalnoise}
\end{figure}

Even with high $\sigma$, the estimator closely tracks the true derivative while rejecting most of the high frequency noise.

\subsection{Time-Varying Optimization with Derivative Estimator}
We now show a simple application of \eqref{eq:dirty_der_recursive} in a time-varying optimization context. We pick a time-varying loss function:
\begin{align*}
    f(x,\theta) &= \frac{1}{2}\lVert x-\theta(t)\rVert^2, \\
    \theta(t) &= \begin{bmatrix}\cos\left(5t-2\right) & \sin\left(5t-2\right) & \cos^2\left(5t-2\right)\end{bmatrix}^T,
\end{align*}
that we minimize using the continuous-time generalization of Newton's method~\eqref{eq:newton_ode} with the time-varying correction~\eqref{eq:newton_correction}.

In the simulation, we consider an \emph{estimated} correction $\widehat u(t) = -\left(\nabla_{xx} f(x, \theta)\right)^{-1}{\nabla_{x\theta}f}(x, \theta)\widehat\theta_1$ where we use our dirty derivative construction \eqref{eq:dirty_der_recursive} to estimate $\dot\theta(t)$ from just the available signal $\theta(t)$. In Fig.~\ref{fig:exloss} we compare our achieved loss over time with the loss achieved using the ideal correction~\eqref{eq:newton_correction}.
\begin{figure}[h]
    \centering
    \vspace{-3mm}
    \includegraphics[width=0.5\textwidth]{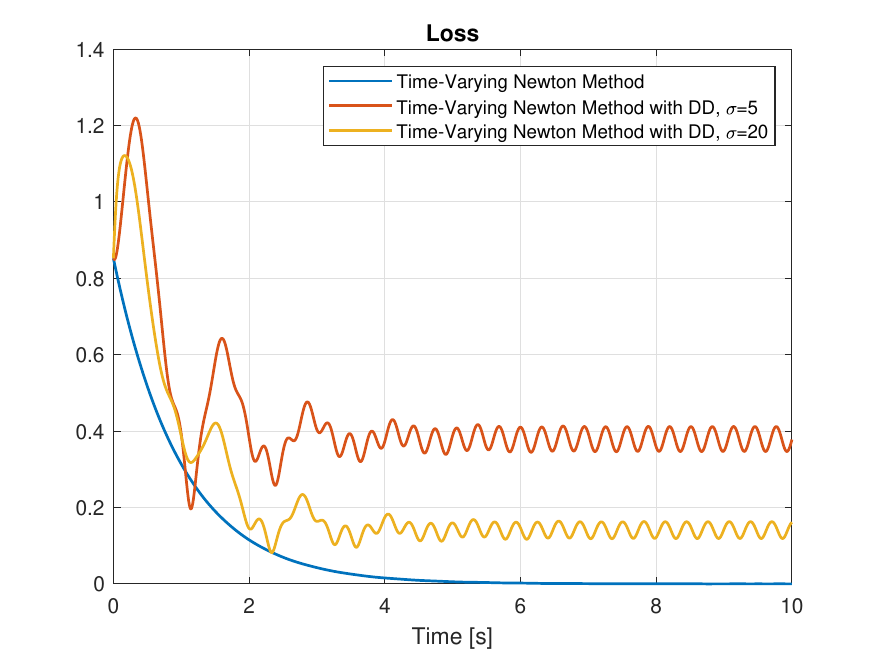}
    \caption{Loss (over time) of the online Newton's method with ground truth (blue) and estimated (red, yellow) time derivative information.}
    \label{fig:exloss}
\end{figure}

Fig. \ref{fig:exloss} shows that our loss trajectory approximates the ideal one, settling closer to zero loss with higher values of $\sigma$. Practically, however, higher values of $\sigma$ tend to amplify the effects of noisy measurements, preventing us from scaling $\sigma$ arbitrarily high. We repeat the experiment with the derivative estimator being fed a signal $\theta(t)$ with each component corrupted by Gaussian noise with zero mean and $0.01$ variance. In Fig.~\ref{fig:exlossnoise} the experiment results show again good robustness to the added noise.
\begin{figure}[h]
    \centering
    \includegraphics[width=0.5\textwidth]{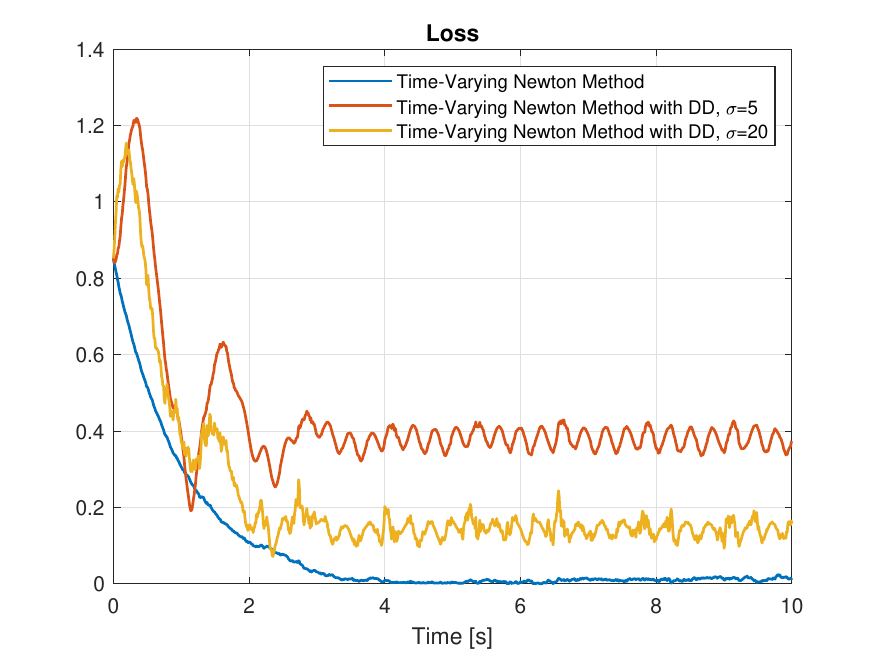}
    \vspace{-5mm}
    \caption{Loss (over time) of the online Newton's method with noise-corrupted ground truth (blue) and estimated (red, yellow) time derivative information.}
    \label{fig:exlossnoise}
    \vspace{-3mm}
\end{figure}

\section{Conclusions}
We have presented a framework to adapt a time-invariant continuous-time optimization algorithm to incorporate a correction term compensating for the time-variation, and showed that popular methods, such as~\eqref{eq:newton_ode}, are suitable for this. We then provided a novel general purpose construction of a derivative estimator with explicit error bounds, and showed that using these estimates in place of true time derivative knowledge in the adapted algorithm results in an IOS system. This provides robustness and convergence results for the tracking error with respect to the true time-varying minimizer of the cost function. Finally, we verified our theory via a set of proof-of-concept simulation experiments.

In future studies we intend to extend the current framework and results to not only the time-varying parameter problem $\theta(t)$ studied here, but to the full closed-loop interaction with a solution dependent parameter vector $\theta(x(t))$.

\appendix
\begin{lemma}\label{lem:single_k_dd_state_space_appendix}
    Consider the LTI SISO dynamical system:
    \begin{equation}\label{eq:single_k_dd_state_space_appendix}
        \begin{cases}
            \dot x(t) &= A_{n,\sigma}x(t) + B_nu(t)\\
            y(t) &= C_nx(t),
        \end{cases}
    \end{equation}
    where $x:\T\to\R^n$, $u, y:\T\to\R$, $\sigma\in\R_{>0}$, and the matrices $A_{n,\sigma}\in\R^{n\times n}$, $B_n\in\R^{n\times 1}$, and $C_n\in\R^{1\times n}$ are:
    \begin{equation*}
        \begin{gathered}
            A_{n,\sigma} = \begin{bmatrix}
                0 & 1 & \dots & 0 \\
                0 & 0 & \dots & 0 \\
                \vdots & \vdots & \ddots & \vdots \\
                0 & 0 & \dots & 1 \\
                -{n\choose 0}\sigma^n & -{n\choose 1}\sigma^{n-1} & \dots & -{n\choose n-1}\sigma
            \end{bmatrix},\\
            B_n = \begin{bmatrix}1 & 0 & \dots & 0\end{bmatrix}^T,\qquad
            C_n = \begin{bmatrix}1 & 0 & \dots & 0\end{bmatrix}.
        \end{gathered}
    \end{equation*}
    Then, there exist $b_{n,\sigma}, d_{n}, a_n \in \R_{>0}$ such that for all $t \in \T$, the output $y$ of~\eqref{eq:single_k_dd_state_space_appendix} satisfies:
    \begin{equation}
        \norm{y(t)} \le \norm{x(0)}b_{n,\sigma}e^{-d_n\sigma t} + \frac{a_n}{\sigma}\norm{u(t)}_\infty.
    \end{equation}
\end{lemma}
\begin{proof}
    We first introduce some notation. Let $\overline\lambda(\cdot)$ and $\underline\lambda(\cdot)$ respectively denote the highest and lowest eigenvalues of a symmetric matrix. Let us define $A_n\in\R^{n\times n}$ as:
    \begin{equation*}
        A_{n} = \begin{bmatrix}
            0 & 1 & 0 & \dots & 0 \\
            0 & 0 & 1 & \dots & 0 \\
            \vdots & \vdots & \vdots & \ddots & \vdots \\
            0 & 0 & 0 & \dots & 1 \\
            -{n\choose 0} & -{n\choose 1} & -{n\choose 2} & \dots & -{n\choose {n-1}}
        \end{bmatrix}.
    \end{equation*}
    Note that because $A_n$ and $A_{n,\sigma}$ are in companion form, their characteristic polynomials are $(\lambda + 1)^n$ and $(\lambda + \sigma)^n$, respectively. Then, let $P_n$ be the positive definite solution of the Lyapunov equality $A_n^TP_n + P_nA_n = -I$. $P_n$ exists and is unique, as all the eigenvalues of $A_n$ are negative.
    
    To analyze the properties of~\eqref{eq:single_k_dd_state_space_appendix}, we introduce the following change of variables:
    \begin{equation}
        \begin{aligned}
            z_1 &= \sigma^{n-1}x_1\\
            z_2 &= \sigma^{n-2}x_2\\
            &\vdots\\
            z_n &= x_n.
        \end{aligned}
    \end{equation}
    Then, the dynamics of $z$ can be written as:
    \begin{equation}
    \begin{aligned}
        \dot z(t) &= \sigma A_n z(t) + \sigma^{n-1} B_n u(t)\\
        &= \sigma \begin{bmatrix}
            0 & 1 & \dots & 0\\
            0 & 0 & \ddots & \vdots\\
            \vdots & \vdots & \ddots & 1\\
            -{n \choose 0} & -{n \choose 1} & \dots & -{n \choose n-1}
        \end{bmatrix} z(t) + \sigma^{n-1} \begin{bmatrix}
            1\\0\\\vdots\\0
        \end{bmatrix} u(t).
    \end{aligned}
    \end{equation}
    Let us now introduce the Lyapunov function $V_n:\R^n\to\R_{\ge 0}$ defined as $V_n(z) = z^TP_nz$ and compute its time derivative (we omit the $t$ argument on $V_n$, $z$, and $u$ for clarity):
    \begin{equation}
        \begin{aligned}
            \dot V_n &= \dot z^T P_n z + z^T P_n \dot z\\
            &= (\sigma A_n z + \sigma^{n-1}B_n u)^TP_n z + z^TP_n(\sigma A_n z + \sigma^{n-1}B_n u)\\
            & = -\sigma z^Tz + 2\sigma^{n-1}uB_n^TP_nz\\
            &\leq -\sigma\norm{z}^2 + \frac{\sigma}{2}\norm{z}^2 + \frac{2}{\sigma}\norm{B_n}^2\norm{P_n}^2\sigma^{2(n-1)}u^2\\
            &\leq -\frac{\sigma}{2\overline\lambda(P_n)}V_n + \frac{2\overline\lambda^2(P_n)}{\sigma}\sigma^{2(n-1)}u^2,
        \end{aligned}
    \end{equation}
    where the third equality holds because $P_n$ is the solution of $A_n^TP_n + P_nA_n$.  Then, by Gr\"onwall's Lemma it holds that:
    \begin{equation}
    \begin{aligned}
        V_n(t) &\leq V_n(0)e^{-\frac{\sigma}{2\overline\lambda_{P_n}}t} + \frac{4\overline\lambda^3_{P_n}}{\sigma^2}\sigma^{2(n-1)}u^2(t)\\
        \implies \norm{z(t)} &\leq \norm{z(0)}\left(\frac{\overline\lambda_{P_n}}{\underline\lambda_{P_n}}\right)^{\frac{1}{2}}e^{-\frac{\sigma}{4\overline\lambda_{P_n}}t} \\&\qquad\qquad + 2\left(\frac{\overline\lambda_{P_n}^3}{\underline\lambda_{P_n}}\right)^{\frac{1}{2}}\sigma^{n-2}\norm{u(t)}_\infty.
    \end{aligned}
    \end{equation}
    From the definition of $y$ and reversing the change of variables we can observe that:
    \begin{equation}
    \begin{gathered}
        \norm{y} = \norm{x_1} = \frac{\norm{z_1}}{\sigma^{n-1}} \leq \frac{\norm{z}}{\sigma^{n-1}},\\
        \norm{z(0)}\leq \max\left(1, \sigma^{n-1}\right)\norm{x(0)},
    \end{gathered}
    \end{equation}
    so:
    \begin{equation}
    \begin{aligned}
        \norm{y(t)} &\leq \max\left(\frac{1}{\sigma^{n-1}}, 1\right)\norm{x(0)}\left(\frac{\overline\lambda_{P_n}}{\underline\lambda_{P_n}}\right)^{\frac{1}{2}}e^{-\frac{\sigma}{4\overline\lambda_{P_n}}t}\\&\qquad+ 2\left(\frac{\overline\lambda_{P_n}^3}{\underline\lambda_{P_n}}\right)^{\frac{1}{2}}\frac{\norm{u(t)}_\infty}{\sigma}.
    \end{aligned}
    \end{equation}
    Then, if we define $b_{n,\sigma}\in\mathcal{KL}_\infty$ and $a_n\in\R_{>0}$ as:
    \begin{equation}
    \begin{gathered}
        b_{n,\sigma}(r,t) = r\max\left(1, \sigma^{n-1}\right)\left(\frac{\overline\lambda_{P_n}}{\underline\lambda_{P_n}}\right)^{\frac{1}{2}}e^{-\frac{\sigma}{4\overline\lambda_{P_n}}t},\\
        a_{n} = 2\left(\frac{\overline\lambda_{P_n}^3}{\underline\lambda_{P_n}}\right)^{\frac{1}{2}},
    \end{gathered}
    \end{equation}
    we can finally write:
    \begin{equation}
        \norm{y(t)} \leq b_{n,\sigma}\left(\norm{x(0)}, t\right) + a_{n}\frac{\norm{u(t)}_\infty}{\sigma}.
    \end{equation}
\end{proof}

\addtolength{\textheight}{-12cm}   




\bibliographystyle{plain}
\bibliography{bibliography}

\begin{thebibliography}{10}

\bibitem{astolfi2018low}
Daniele Astolfi, Lorenzo Marconi, Laurent Praly, and Andrew~R Teel.
\newblock Low-power peaking-free high-gain observers.
\newblock {\em Automatica}, 98:169--179, 2018.

\bibitem{cisneros2022contraction}
Pedro Cisneros-Velarde and Francesco Bullo.
\newblock A contraction theory approach to optimization algorithms from acceleration flows.
\newblock In {\em International Conference on Artificial Intelligence and Statistics}, pages 1321--1335. PMLR, 2022.

\bibitem{dall2020optimization}
Emiliano Dall'Anese, Andrea Simonetto, Stephen Becker, and Liam Madden.
\newblock Optimization and learning with information streams: Time-varying algorithms and applications.
\newblock {\em IEEE Signal Processing Magazine}, 37(3):71--83, 2020.

\bibitem{fazlyab2017prediction}
Mahyar Fazlyab, Santiago Paternain, Victor~M Preciado, and Alejandro Ribeiro.
\newblock Prediction-correction interior-point method for time-varying convex optimization.
\newblock {\em IEEE Transactions on Automatic Control}, 63(7):1973--1986, 2017.

\bibitem{gong2016time}
Ping Gong, Fei Chen, and Weiyao Lan.
\newblock Time-varying convex optimization for double-integrator dynamics over a directed network.
\newblock In {\em 2016 35th Chinese Control Conference (CCC)}, pages 7341--7346. IEEE, 2016.

\bibitem{hassan2012algorithm}
Hassan~Zohair Hassan, AA~Mohamad, and GE~Atteia.
\newblock An algorithm for the finite difference approximation of derivatives with arbitrary degree and order of accuracy.
\newblock {\em Journal of Computational and Applied Mathematics}, 236(10):2622--2631, 2012.

\bibitem{isidori1986sufficient}
Alberto Isidori, CH~Moog, and A~De~Luca.
\newblock A sufficient condition for full linearization via dynamic state feedback.
\newblock In {\em 1986 25th IEEE Conference on Decision and Control}, pages 203--208. IEEE, 1986.

\bibitem{johnson2005pid}
Michael~A Johnson and Mohammad~H Moradi.
\newblock {\em PID control}.
\newblock Springer, 2005.

\bibitem{khalil2014high}
Hassan~K Khalil and Laurent Praly.
\newblock High-gain observers in nonlinear feedback control.
\newblock {\em International Journal of Robust and Nonlinear Control}, 24(6):993--1015, 2014.

\bibitem{khalil1992nonlinear}
H.K. Khalil.
\newblock {\em Nonlinear Systems}.
\newblock Macmillan Publishing Company, 1992.

\bibitem{khan1999closed}
Ishtiaq~Rasool Khan and Ryoji Ohba.
\newblock Closed-form expressions for the finite difference approximations of first and higher derivatives based on taylor series.
\newblock {\em Journal of Computational and Applied Mathematics}, 107(2):179--193, 1999.

\bibitem{krishnan2012selection}
Sunder~Ram Krishnan and Chandra~Sekhar Seelamantula.
\newblock On the selection of optimum savitzky-golay filters.
\newblock {\em IEEE transactions on signal processing}, 61(2):380--391, 2012.

\bibitem{loria2015observers}
Antonio Lor{\'\i}a.
\newblock Observers are unnecessary for output-feedback control of {L}agrangian systems.
\newblock {\em IEEE Transactions on Automatic Control}, 61(4):905--920, 2015.

\bibitem{muehlebach2019dynamical}
Michael Muehlebach and Michael Jordan.
\newblock A dynamical systems perspective on nesterov acceleration.
\newblock In {\em International Conference on Machine Learning}, pages 4656--4662. PMLR, 2019.

\bibitem{patidar2005use}
Kailash~C Patidar.
\newblock On the use of nonstandard finite difference methods.
\newblock {\em Journal of Difference Equations and Applications}, 11(8):735--758, 2005.

\bibitem{popkov2005gradient}
A~Yu Popkov.
\newblock Gradient methods for nonstationary unconstrained optimization problems.
\newblock {\em Automation and Remote Control}, 66:883--891, 2005.

\bibitem{rahili2015distributed}
Salar Rahili and Wei Ren.
\newblock Distributed continuous-time convex optimization with time-varying cost functions.
\newblock {\em IEEE Transactions on Automatic Control}, 62(4):1590--1605, 2017.

\bibitem{romero2022ode}
Orlando Romero, Mouhacine Benosman, and George~J Pappas.
\newblock {ODE} discretization schemes as optimization algorithms.
\newblock In {\em 2022 IEEE 61st Conference on Decision and Control (CDC)}, pages 6318--6325. IEEE, 2022.

\bibitem{simonetto2020time}
Andrea Simonetto, Emiliano Dall'Anese, Santiago Paternain, Geert Leus, and Georgios~B Giannakis.
\newblock Time-varying convex optimization: Time-structured algorithms and applications.
\newblock {\em Proceedings of the IEEE}, 108(11):2032--2048, 2020.

\bibitem{simonetto2016class}
Andrea Simonetto, Aryan Mokhtari, Alec Koppel, Geert Leus, and Alejandro Ribeiro.
\newblock A class of prediction-correction methods for time-varying convex optimization.
\newblock {\em IEEE Transactions on Signal Processing}, 64(17):4576--4591, 2016.

\bibitem{sioshansi2017optimization}
Ramteen Sioshansi, Antonio~J Conejo, et~al.
\newblock Optimization in engineering.
\newblock {\em Cham: Springer International Publishing}, 120, 2017.

\bibitem{Sontag2008}
Eduardo~D. Sontag.
\newblock {\em Input to State Stability: Basic Concepts and Results}, pages 163--220.
\newblock Springer Berlin Heidelberg, Berlin, Heidelberg, 2008.

\bibitem{sontag1999notions}
Eduardo~D Sontag and Yuan Wang.
\newblock Notions of input to output stability.
\newblock {\em Systems \& Control Letters}, 38(4-5):235--248, 1999.

\bibitem{yang2016tracking}
Tianbao Yang, Lijun Zhang, Rong Jin, and Jinfeng Yi.
\newblock Tracking slowly moving clairvoyant: Optimal dynamic regret of online learning with true and noisy gradient.
\newblock In {\em International Conference on Machine Learning}, pages 449--457. PMLR, 2016.

\end{thebibliography}

\end{document}